\documentclass[14paper,10pt]{amsart}

\usepackage{amsmath} 
\usepackage{amssymb} 
\usepackage{hyperref}
\usepackage{marginnote}
\usepackage{enumitem}
\usepackage{tikz}
\usepackage[normalem]{ulem}
\usepackage{centernot}

\usetikzlibrary{arrows,automata,positioning,backgrounds}

\tikzstyle{vecArrow} = [thick, decoration={markings,mark=at position
   1 with {\arrow[semithick]{open triangle 60}}},
   double distance=1.4pt, shorten >= 5.5pt,
   preaction = {decorate},
   postaction = {draw,line width=1.4pt, white,shorten >= 4.5pt}]
\tikzstyle{innerWhite} = [semithick, white,line width=1.4pt, shorten >= 4.5pt]

\DeclareMathOperator*{\esssup}{ess\,sup}
\newcommand{\calK}{\mathcal{K}}

\newcommand{\calL}{\mathcal{L}}
\newcommand{\Rpn}{\R_{0}^{+}}

\newcommand{\C}{\mathbb{C}}
\newcommand{\Cp}{\C_{+}}

\newcommand{\N}{\mathbb{N}}
\newcommand{\R}{\mathbb{R}}

\newcommand{\Z}{\mathbb{Z}}
\renewcommand{\Re}{\operatorname{Re}}

\newcommand{\vertiii}[1]{{\left\vert\kern-0.25ex\left\vert\kern-0.25ex\left\vert #1 
    \right\vert\kern-0.25ex\right\vert\kern-0.25ex\right\vert}}

\def\text{\textrm}

 \newtheorem{theorem}{Theorem}[section]
 
 \newtheorem{lemma}[theorem]{Lemma}
 \newtheorem{proposition}[theorem]{Proposition}
 \theoremstyle{definition}
 \newtheorem{definition}[theorem]{Definition}
 \theoremstyle{remark}
 \newtheorem{rem}[theorem]{Remark}
 \newtheorem{example}[theorem]{Example}
 \numberwithin{equation}{section}

\begin{document}

\title[Infinite-dimensional ISS and Orlicz spaces]{Infinite-dimensional input-to-state stability and Orlicz spaces}

%
   %

\date{}

%
\author[B.~Jacob]{Birgit Jacob}
\address{Functional analysis group, School of Mathematics and Natural Sciences,
        University of Wuppertal, D-42119 Wuppertal, Germany} 
\email{bjacob@uni-wuppertal.de}
\author[R.~Nabiullin]{Robert Nabiullin}
\address{Functional analysis group, School of Mathematics and Natural Sciences,
        University of Wuppertal, D-42119 Wuppertal, Germany} 
\email{nabiullin@math.uni-wuppertal.de}
\author[J.R.~Partington]{Jonathan R.~Partington}
\address{School of Mathematics, University of Leeds, Leeds LS2 9JT, Yorkshire, United Kingdom}
\email{j.r.partington@leeds.ac.uk}
\author[F.L.~Schwenninger]{Felix L.~Schwenninger}
\address{Department of Mathematics, Center for Optimization and Approximation, University of Hamburg, 
Bundesstra\ss e 55, D-20146 Hamburg, Germany}
\email{felix.schwenninger@uni-hamburg.de}
\thanks{The second author is supported by Deutsche Forschungsgemeinschaft (Grant JA 735/12-1)}

\begin{abstract}
In this work, the relation between input-to-state stability and integral input-to-state stability is studied for linear infinite-dimensional systems with an unbounded control operator. Although a special focus is laid on the case $L^{\infty}$, general function spaces are considered for the inputs. We show that integral input-to-state stability can be characterized in terms of input-to-state stability with respect to Orlicz spaces. Since we consider linear systems, the results can also be formulated in terms of admissibility.
 For parabolic diagonal systems with scalar inputs, both stability notions with respect to $L^\infty$ are equivalent.
\end{abstract}

 

\maketitle

\section{Introduction}
In systems and control theory, the question of stability is a fundamental issue. 
Let us consider the situation where the relation between the input (function) $u$ and the state $x$  is governed by the autonomous equation
\begin{equation}\label{intro:eq1}
\dot{x} = f(x,u), \quad x(0)=x_{0}.
\end{equation}
One can then distinguish between \textit{external stability}, that is, stability with respect to the input $u$,  and {\it internal stability}, i.e.~when $u=0$. For the moment, $f$ is assumed to map from $\R^{n}\times \R^m$ to $\R^n$, and to be such that solutions $x$ 
exist on $[0,\infty)$ for all inputs $u$ in a function space $Z$. Already from this very general view-point, it seems clear that stability notions may strongly depend on the specific choice of $Z$ (and its norm).
The concept
 of \textit{input-to-state stability}
(ISS) combines both external and internal stability in one notion. If $Z$ is chosen to be $L^{\infty}(0,\infty;U)$, $U=\R^{m}$, a system is called ISS (with respect to $L^{\infty}$) if 
there exist functions $\beta\in \mathcal{KL}$, $\gamma\in\mathcal{K}$ such that
\[\|x(t)\| \leq \beta(\|x_{0}\|,t) + \gamma(\esssup_{s\in[0,t]}\|u(s)\|_{U}),\]
for all $t>0$ and $u\in Z$. Here the sets $\mathcal{KL}$ and $\mathcal{K}$ refer to the classic comparison functions from nonlinear systems theory, see Section \ref{sec:2}. 
Introduced by E.~Sontag in 1989 \cite{Sontag89ISS}, ISS has been  intensively studied in the past decades; see \cite{sontagISS} for a survey.\\
A related stability notion is {\it integral input-to-state stability} (iISS)  \cite{Sontag98, AnSoWa00}, which means that for some $\beta\in \mathcal{KL}$, $\theta\in\mathcal{K}_{\infty}$ and $\mu\in\mathcal{K}$, 
\begin{equation}\label{intro:eq2}
\|x(t)\| \leq \beta(\|x_{0}\|,t) +\theta\left(\int_{0}^{t}\mu(\|u(s)\|)_{U})\,ds\right), 
\end{equation}
for all $t>0$ and {$u\in Z=L^{\infty}(0,\infty;U)$}. {This property differs from ISS in the sense that it allows for unbounded inputs $u$ that have ``finite energy'', see \cite{Sontag98}. Many practically relevant systems are iISS whereas they are not ISS, see e.g.~\cite{MiI14,MiI16} for a detailed list. However, }for linear systems, i.e., $f(x,u)=Ax+Bu$ with matrices $A$ and $B$, iISS is  equivalent to ISS. 
To some extent, this observation marks the starting point of this work.\par
In contrast to the well-established theory for finite-dimensions, a more intensive study of 
(integral) input-to-state stability for infinite-di\-men\-sional systems has  only begun recently. We refer to \cite{DaM13,DaM13b,JaLoRy08,Krstic16MTNS,KK2016IEEETac,Log13,Mir16,MiI14,MiI15b,MiW15,MiI16}. 
By nature, in the infinite-dimensional setting, the stability notions from finite-dimensions are more subtle.  We refer to \cite{MiW16} for a listing of failures of equivalences around ISS known from finite-dimensional systems. 
In most of the mentioned infinite-dimensional references, 
systems of the form \eqref{intro:eq1} with $f \colon X \times U\rightarrow X $ and Banach spaces $X$ and $U$ are considered. 
For linear equations, this setting corresponds to evolution equations of the form
\begin{equation}\label{eq:sys1}\dot{x}(t)=Ax(t)+Bu(t), \quad x(0)=x_{0},\end{equation}
where $B$ is a bounded control operator (note that for fixed $t$, $x(t)=x(t,\cdot)$ is a function and $\dot{x}$ denotes the time-derivative).
 Analogously to finite-dimensions, in this case,  ISS and iISS are known to be equivalent, see e.g., \cite[Corallary 2]{MiI16} and Proposition \ref{prop1b} below. However, concerning applications the requirement of bounded control operators $B$ is rather restrictive. Typical examples for systems which  only allow for a formulation with an unbounded $B$ are boundary control systems. It is clear that such 
phenomena cannot occur for linear systems in finite-dimensions. 
 
The main point of this paper is to relate and characterize (integral) input-to-state stability for linear, infinite-dimensional systems with unbounded control operators, i.e.\ systems of the form \eqref{eq:sys1} with unbounded operators $B$.  This is done by using the notion of \textit{admissibility}, \cite{salamon,Weiss89ii}, which also reveals the connection of the mentioned stability types with the boundedness of the linear mapping
\[Z\to X, \qquad u\mapsto x(t)\]
(for $x_{0}=0$). It is not surprising that the choice of topology for $Z$, the space of inputs $u$, is crucial here. However, looking at \eqref{intro:eq2} for $x_{0}=0$, it is not clear how the right-hand side could define a norm {for general functions $\mu$ and $\theta$.
 The question of the right norm for $Z$ motivates one to study ISS and iISS with respect to general spaces $Z$ -- not only $Z=L^{\infty}=L^{\infty}(0,\infty;U)$.  For the precise definition of these notions, we refer to Section \ref{sec:2}. We show that $Z$-ISS and $Z$-iISS are equivalent for $Z=L^{p}=L^{p}(0,\infty;U)$, $p\in[1,\infty)$. However, it turns out that this paves the way to characterize $L^{\infty}$-iISS in terms of ISS. More precisely, we will show that $L^\infty$-iISS is equivalent to  ISS with respect to some {\it Orlicz space}.} This is one of the main results of this work. Orlicz spaces (or Orlicz--Birnbaum spaces) appear naturally as generalizations of $L^{p}$-spaces and ISS with respect to such spaces can thus be seen as a generalization of classical stability notions. {Other choices for general input functions have been made in the literature -- like admissibility with respect to Lorentz spaces \cite{Haak2012, Wy10} or $Z$-ISS with $Z$ being a Sobolev space \cite{JPP14, MiI15b}. } \\
{As we will see, it is plain that  $Z$-iISS always implies $Z$-ISS for linear systems}. The converse direction, for $Z=L^\infty$, remains open in general. 
It is known that ISS is equivalent to admissibility (together with exponential stability). We will show that  $L^{\infty}$-iISS in fact implies \textit{zero-class admissibility} \cite{ZeroclassJPP,XuZeroclass}, which is slightly stronger than admissibility, see Proposition \ref{prop:zeroclass}.
In Table \ref{table:overview}, the relation of $L^{\infty}$-ISS and $L^{\infty}$-iISS in the various above-mentioned settings is depicted schematically.

In Section \ref{sec:2}, we will discuss the setting and formally introduce the stability notions mentioned above. This includes a general abstract definition of ISS, iISS and admissibility with respect to some function space $Z$. Furthermore, we will give some basic facts about their relation. 

Section \ref{sec:Orlicz} deals with the characterization of ISS and iISS in terms of Orlicz-space-admissibility. As a main result, we show that $L^\infty$-iISS is equivalent to ISS with respect to some Orlicz space $E_{\Phi}$, where $\Phi$ denotes a Young function, Theorem \ref{thm:main}. Moreover, we show that ISS with respect to an Orlicz space is a natural generalization of classic $L^p$-ISS that ``interpolates'' the notions of $L^1$- and $L^{\infty}$-ISS, Theorems \ref{thm:genLp} and \ref{thm:main2}.

In Section \ref{sec:diag.syst}, we consider parabolic diagonal systems with scalar input. More precisely, we assume that $A$ possesses a Riesz  basis of eigenvectors with eigenvalues lying in a sector in the open left half-plane.  For this class of systems we show that $L^{\infty}$-ISS implies ISS with respect to some Orlicz space and thus, by the results of Section \ref{sec:Orlicz}, the equivalence between iISS and ISS, known in finite dimensions, holds for this class of systems. Moreover, it turns out that any linear, bounded operator from $U$ to the extrapolation space $X_{-1}$ is $L^{\infty}$-admissible, which yields a characterization of ISS. The results of this section partially generalize results that were already indicated in \cite{CDCJNPS16}.

 We illustrate the obtained results by examples in Section \ref{sec:ex}. In particular, we present a parabolic diagonal system which is $L^{\infty}$-ISS, but not $L^{p}$-ISS for any $p\in[1,\infty)$. Finally, we conclude by drawing a connection between the question whether $L^{\infty}$-ISS implies $L^{\infty}$-iISS and a problem due to G.~Weiss.

\begin{table}
\center
\def\arraystretch{1.5}%
\begin{tabular}{c||c|c|c}
& \parbox[t]{1.9cm}{Eq.~\eqref{eq:sys1},\\ $B$ bounded}&\parbox[t]{2.3cm}{Eq.~\eqref{eq:sys1}, \\$B$ unbounded\vspace{0.2cm} } &\parbox[t]{2cm}{ Eq.~\eqref{intro:eq1},\\ $f$ nonlinear}\\
\hline\hline
$\mathrm{dim}\, X<\infty$&ISS $\Longleftrightarrow$ iISS&ISS $\iff$ iISS&ISS $\Longrightarrow \atop{\centernot\Longleftarrow}$ iISS\\
\hline
$\mathrm{dim}\, X=\infty$&ISS $\iff$ iISS& ISS $\impliedby\atop \left(\stackrel{?}{\Longrightarrow}\right)$ iISS& not clear \\
\hline
\end{tabular}
\medskip

\label{table:overview}
\caption{The relation between ISS and iISS (with respect to $L^{\infty}$) in various settings.}
\end{table}

\section{Stability notions for infinite-dimensional systems}\label{sec:2}
\subsection{The setting and definitions}\label{sec:2.1}
In this article we study systems $\Sigma(A,B)$ of the following form
\begin{equation}\label{eqn1}
\dot{x}(t)=Ax(t)+Bu(t), \qquad x(0)=x_0, \quad t\ge 0,
\end{equation}
where $A$ generates a $C_0$-semigroup $(T(t))_{t\ge 0}$ on a Banach space $X$ and $B$ is a linear and bounded operator from a Banach space $U$ to the extrapolation space $X_{-1}$. 
Note that $B$ is possibly unbounded from $U$ to $X$.
 Here $X_{-1}$ is the completion of $X$ with respect to the norm 
\[ \|x\|_{X_{-1}}= \|(\beta -A)^{-1}x\|_X,\]
for some $\beta\in \rho(A)$, the resolvent set of $A$. It can be shown that the semigroup  $(T(t))_{t\ge 0}$ possesses a unique extension to a $C_{0}$-semigroup  $(T_{-1}(t))_{t\ge 0}$ on $X_{-1}$ with generator $A_{-1}$, which is an extension of $A$.
Thus we may consider equation \eqref{eqn1} on the Banach space $X_{-1}$ and therefore for $u\in L^1_{loc}(0,\infty;U)$,  the \textit{(mild) solution of  \eqref{eqn1}} is given by the variation of parameters formula
\begin{equation} \label{eqn2}
x(t)=T(t)x_0+\int_0^t T_{-1}(t-s)B u(s) \, ds, \qquad t\ge 0.
\end{equation}

In this paper, we will consider the following types of function spaces. 
For a Banach space $U$, let $Z\subseteq L_{loc}^{1}(0,\infty;U)$ be such that for all $t>0$
\begin{enumerate}[label=(\alph*)]
\item  $Z(0,t;U):=\{f\in Z \mid f|_{[t,\infty)} = 0\}$ becomes a Banach space of functions on the interval $(0,t)$ with values in $U$ (in the sense of equivalence classes w.r.t.\ equality almost everywhere),
\item $Z(0,t;U)$ is continuously embedded in $L^1(0,t; U)$,  that is, there exists $\kappa(t)>0$ such that for all $f\in Z(0,t;U)$ it holds that $f\in L^{1}(0,t;U)$ and 
\[ \|f\|_{L^{1}(0,t;U)}\leq \kappa(t)\|f\|_{Z(0,t;U)}. \]
\item  For $u\in Z(0,t;U)$ and  $s>t$ we have $\|u\|_{Z(0,t;U)}=\| u \|_{Z(0,s;U)}$.
\item {  $Z(0,t;U)$ is invariant under the left-shift and reflection, i.e., $S_{\tau}Z(0,t;U)\subset Z(0,t;U)$ and $R_{t}Z(0,t;U)\subset Z(0,t;U)$, where 
$$S_{\tau}u=u(\cdot+\tau), \quad R_{t}u=u(t-\cdot),$$ and $\tau>0$. Furthermore, $\|S_{\tau}\|_{Z(0,t;U)}\leq 1$  and $R_{t}$ is isometric.}
\item For all $u\in Z$ and $0<t<s$ it holds that $u|_{(0,t)}\in Z(0,t;U)$ and
\[\|u|_{(0,t)}\|_{Z(0,t;U)} \leq \|u|_{(0,s)}\|_{Z(0,s;U)}.\]
\end{enumerate}
 If additionally we have in (b) that
 \begin{equation}\tag{B}\label{cond:B}
 \kappa(t)\to 0, \quad  \textrm{ as } t\searrow 0,
 \end{equation}
 then we say that $Z$ satisfies {\em condition (B)}.

For example,  $Z=L^{p}$ refers to the spaces $L^p(0,t;U)$, $t>0$, for fixed $1\leq p\leq \infty$ and $U$. 
Other examples can be given by Sobolev spaces and the  Orlicz spaces $L_\Phi(0,t; U)$ and $E_\Phi(0,t;U)$, see the appendix. 
 If $p>1$ (including $p=\infty$) and $\Phi$ is a Young function, then $L^p$,  $E_\Phi$ and $L_\Phi$ satisfy Condition (B),
thanks to H\"older's inequality. Clearly, $L^1$ does not satisfy condition (B).

 In general, the state $x(t)$ given by  \eqref{eqn2} lies in $X_{-1}$  for $u\in L_{loc}^{1}$ and $t>0$. The notion of \textit{admissibility} ensures that  indeed $x(t)\in X$.
\begin{definition}\label{def:iISS}
 We call the  system $\Sigma(A,B)$  {\em admissible with respect to $Z$} (or {\em $Z$-admissible}), if 
\begin{equation}\label{eq:Admissibility}
\int_{0}^{t}T_{-1}(s)Bu(s)\, ds \in X
\end{equation}
\end{definition}
for all $t>0$ and $u\in Z(0,t;U)$. If $\Sigma(A,B)$ is admissible with respect to $Z$, then all mild solutions  \eqref{eqn2} are in $X$ and by the closed graph theorem there exists a constant $c(t)$ (take the infimum over all possible constants) such that
\begin{equation}\label{eq:adm2}
	\left\|\int_{0}^{t}T_{-1}(s)Bu(s)\, ds\right\| \leq c(t) \|u\|_{Z(0,t;U)}.
\end{equation}
Moreover, it is easy to see that $\Sigma(A,B)$ is admissible if \eqref{eq:Admissibility} holds for one $t>0$.

\begin{definition}
We call the  system $\Sigma(A,B)$ {\em infinite-time admissible with respect to $Z$} (or {\em $Z$-infinite-time admissible}), if the system is admissible with respect to $Z$ and $c_\infty:=\sup_{t>0} c(t)$ is finite. We call the  system $\Sigma(A,B)$  {\em zero-class admissible with respect to $Z$} (or {\em $Z$-zero-class admissible}), if it is admissible with respect to $Z$ and $\lim_{t \to 0} c(t) = 0$.
\end{definition}

\begin{rem}\label{remadm}
Clearly, zero-class admissibility and infinite-time admissibility imply admissibility respectively.\newline
\end{rem}

Since $Z\subseteq L_{loc}^{1}(0,\infty;U)$, for any $u\in Z$ and any initial value $x_{0}$, the mild solution $x$ of \eqref{eqn1} is continuous as function from $[0,\infty)$ to $X_{-1}$. Next we show that zero-class admissibility guarantees that $x$ even lies in $C(0,\infty;X)$.
\begin{proposition}\label{prop:zeroclass2}
If $\Sigma(A,B)$ is $Z$-zero-class admissible, then for every $x_0 \in X$ and every $u \in Z$ the mild solution of \eqref{eqn1}, given by \eqref{eqn2}, satisfies $x \in C([0,\infty);X)$. 
\end{proposition}
\begin{proof}
Since $x$ is given by \eqref{eqn2}, it suffices to consider the case $x_{0}=0$. Let $u\in Z$. We have to show that $t\mapsto \Phi_{t}u:=\int_{0}^{t}T_{-1}(s)Bu(s)\,ds$  is continuous. The proof is divided into two steps.\\
First, note that $t\mapsto \Phi_{t}u$ is right-continuous on $[0,\infty)$. In fact, by
\begin{align*}
 \Phi_{t+h}u - \Phi_{t}u = {}& T(t)\int_{0}^{h}T_{-1}(s)Bu(s+t)\,ds,
\end{align*}
$h>0$,
and $Z$-zero-class admissibility, it follows that 
\[\|\Phi_{t+h}u - \Phi_{t}u\| \leq c(h)\|T(t)\|\|u(\cdot+t)\|_{Z(0,h;U)}\to 0\]
for $h\searrow 0$ (where we used properties (d), (e) of $Z$).\\
Second, we show that $t\mapsto\Phi_{t}$ is left-continuous on $(0,\infty)$.  Since $(\Phi_{t}-\Phi_{t-h})u=(\Phi_{t}-\Phi_{t-h})u|_{(0,t)}$, we can assume that $u\in Z(0,t;U)$. Clearly,
\[(\Phi_{t}-\Phi_{t-h})u = T(t-h)\int_{0}^{h}T_{-1}(s)Bu(s+t-h)\,ds.\]
It follows that
\begin{align*}
\left \|\int_{0}^{h}T_{-1}(s)Bu(s+t-h)\,ds \right \|\leq{}& c(h)\|u(\cdot+t-h)\|_{Z(0,h;U)}\\ 
\leq{}& c(h)\|u(\cdot+t-h)\|_{Z(0,t;U)}\\
\leq{}& c(h)\|u\|_{Z(0,t;U)}\stackrel{h\searrow0}{\longrightarrow}0,
\end{align*}
where the last two inequalities hold by properties (e) and (d) of $Z$. Since $(T(t))_{t \geq 0}$ is uniformly bounded on compact intervals, we conclude that $ \|\Phi_{t+h}u - \Phi_{t}u \|\to 0$ as $h\rightarrow 0$.
\end{proof}
\begin{rem}
  If $\Sigma(A,B)$ is admissible with respect to $L^p$, $1\leq p < \infty$, then, by H\"older's inequality,  $\Sigma(A,B)$ is $L^q$-zero-class admissible for any $q>p$. Thus, Proposition \ref{prop:zeroclass2} implies that the mild solution of \eqref{eqn1} lies in $C(0,\infty;X)$ for all $u\in L^{q}$. Moreover, this continuity even holds for $u\in L^{p}$, which was already shown by G.~Weiss in his seminal paper \cite[Prop.~2.3]{Weiss89ii} on admissible control operators. However, there, a direct, but similar proof is used without using the notion of zero-class admissibility. 
  \newline
  As stated in \cite[Problem 2.4]{Weiss89ii}, it is an interesting open problem whether the continuity of $x$ is implied by $L^\infty$-admissibility. By Proposition \ref{prop:zeroclass2},  the answer is `yes' in the case of $L^{\infty}$-zero-class admissibility. See also Section \ref{sec:Weissproblem}.
\end{rem}

To introduce input-to-state stability, we will need the following well-known function classes from Lyapunov theory. Here, $\Rpn$ denotes the set of nonnegative real numbers.
\begin{align*} 
\calK ={}& \{\mu \colon \Rpn\rightarrow \Rpn \:|\: \mu(0)=0,\, \mu\text{ continuous, }  \text{strictly increasing}\},\\
 \calK_\infty ={}& \{\theta\in\calK \:|\:  \lim_{x\to\infty} \theta(x)=\infty\},\\
\mathcal{L}={}&\{\gamma \colon \Rpn\rightarrow\Rpn\:|\:\gamma \text{ continuous, } \text{strictly decreasing,} \lim_{t\to\infty}\gamma(t)=0 \},\\
\mathcal{KL} = {}&\{ \beta \colon (\Rpn)^{2}\rightarrow\Rpn\: | \: \beta(\cdot,t)\in\calK\ \forall t \geq 0 \text{ and }\beta(s,\cdot)\in\mathcal{L}\ \forall s> 0\}.
\end{align*}

\begin{definition} 
The system $\Sigma(A,B)$  is called {\em input-to-state stable with respect to  $Z$} (or {\em  $Z$-ISS}), if there exist functions $\beta\in \mathcal{KL}$ and $ \mu\in {\calK}_\infty$ such that for every $t\ge 0$, $x_0\in X$ and  $u\in Z(0,t;U)$
\begin{enumerate}[label=(\roman*)]
\item  $x(t)$ lies in $X$  and
\item $
\left\| x(t)\right\| \le \beta(\|x_0\|,t)+  \mu (\|u\|_{Z(0,t;U)}).$
\end{enumerate}

 The system $\Sigma(A,B)$  is called  {\em integral input-to-state stable with respect to  $Z$} (or {\em $Z$-iISS}), if   there exist  functions $\beta\in \mathcal{KL}$, $\theta\in {\calK}_\infty$ and $\mu \in {\calK}$ such that for every $t\ge 0$, $x_0\in X$ and  $u\in Z(0,t;U)$
\begin{enumerate}[label=(\roman*)]
\item $x(t)$ lies in $X$ and 
\item $\displaystyle
\left\| x(t)\right\| \le \beta(\|x_0\|,t)+ \theta \left(\int_0^t \mu (\|u(s)\|_{U}) \, ds\right).$
\end{enumerate}

 The system $\Sigma(A,B)$  is called  {\em uniformly bounded energy bounded state  with respect to  $Z$} (or {\em $Z$-UBEBS}), if  there exist  functions  $\gamma, \theta\in {\calK}_\infty$, $\mu \in {\calK}$ and a constant $c >0$ such that for every $t\ge 0$, $x_0\in X$ and  $u\in Z(0,t;U)$
\begin{enumerate}[label=(\roman*)]
\item $x(t)$ lies in $X$  and 
\item $\displaystyle
\left\| x(t)\right\| \le \gamma(\|x_0\|)+ \theta \left(\int_0^t \mu (\|u(s)\|_{U}) \, ds\right) +c.$
\end{enumerate}
\end{definition}

\begin{rem}\label{remISS}
\begin{enumerate}
\item
By the inclusion of $L^p$ spaces on bounded intervals we obtain that $L^p$-ISS ($L^p$-iISS, $L^ p$-UBEBS)  implies  $L^q$-ISS ($L^q$-iISS, $L^ q$-UBEBS)  for all $1 \leq p < q \leq \infty$. Further the inclusions $L^\infty \subseteq E_\Phi \subseteq L_\Phi \subseteq L^1$ and $Z\subseteq L^1_{loc}$ yield a corresponding chain of implications of  ISS, iISS and UBEBS.
\item {Note that in general the integral $\int_0^t \mu (\|u(s)\|_{U}) \, ds$ in the inequalities defining $Z$-iISS and $Z$-UBEBS may be infinite. In that case, the inequalities hold trivially. This indicates that the major interest in iISS and UBEBS lies in the case $Z=L^{\infty}$, in which the integral is always finite.}
\end{enumerate}
\end{rem}

\subsection{Relations between the stability notions}

Recall that the semigroup $(T(t))_{t\ge 0}$ is called exponentially stable, if there exist constants $M,\omega>0$ such that
\begin{equation}\label{eqnexp}
\|T(t)\|\le M{\rm{e}}^{-\omega t}, \quad t\ge 0.
\end{equation}
{
\begin{lemma}\label{lem2}
 Let $(T(t))_{t\ge 0}$ be exponentially stable and $\Sigma(A,B)$ be $Z$-admissible. Then the following holds.
 \begin{enumerate}[label=(\roman*)]
 \item\label{lem2:it1} $\Sigma(A,B)$ is infinite-time $Z$-admissible.
 
 \item\label{lem2:it2}
  $\Sigma(A,B)$ is $Z$-iISS if and only if there exist $\theta \in \mathcal K_\infty$ and $\mu \in \mathcal K$ such that for every $u\in Z(0,1;U)$,
\begin{equation}\label{eq:lemma14}
\left\|\int_0^1 T_{-1}(s) B u(s) \,ds \right\| \le \theta \left(\int_0^1 \mu (\|u(s)\|_{U}) \, ds\right).
\end{equation}
Moreover, if (\ref{eq:lemma14}) holds, then  $\Sigma(A,B)$ is $Z$-iISS with the same choice of $\mu$.
\end{enumerate}
\end{lemma}

\begin{proof} Clearly, in \ref{lem2:it2} we only have to show that the condition for $Z$-iISS is sufficient.
Therefore, in both \ref{lem2:it1} and  \ref{lem2:it2} it suffices to show that there exists $C>0$ such that for any $t>0$ and $u\in Z(0,t;U)$, there exists $\tilde{u}\in Z(0,1;U)$ with
  \[\left\|\int_0^t T_{-1}(s) B u(s) \,ds \right\| \le C \left\|\int_0^1 T_{-1}(s) B \tilde{u}(s) \,ds \right\|\]
such that $\|\tilde{u}\|_{Z(0,1;U)}\leq \|u\|_{Z(0,t;U)}$  and $\int_0^1 \mu (\|\tilde{u}(s)\|_{U}) \, ds \leq \int_0^t \mu (\|u(s)\|_{U}) \, ds$ for any $\mu\in\mathcal{K}$. Without loss of generality, we assume that $t\in\N$, otherwise extend $u$ suitably by the zero-function. 
By splitting the integral, substitution and the fact that $\Sigma(A,B)$ is $Z$-admissible, we get for $u\in Z(0,t;U)$,
\begin{align*}
 \left \Vert \int_0^t T_{-1} (s) B u(s) \, ds \right \Vert ={}&\left \Vert \sum_{k=0}^{t-1} \int_k^{k+1}T_{-1}(s) B u(s) \, ds \right \Vert\\
= {}&\left \Vert \sum_{k=0}^{t-1} T(k) \int_0^1 T_{-1}(s)B u(s+k) \, ds \right \Vert \\ \leq {}&  \sum_{k=0}^{t-1} \left \Vert T(k)  \right \Vert \max_{k=0,..,t-1} \left \Vert  \int_0^1 T_{-1}(s)B u(s+k) \, ds  \right \Vert\\ \leq {}&C \cdot \max_{k=0,..,t-1} \left \Vert  \int_0^1 T_{-1}(s)B u(s+k) \, ds  \right \Vert ,
 \end{align*}
where $C<\infty$ only depends on the exponentially stable semigroup $(T(t))_{t\geq 0}$. 
Choose $\tilde{u}=u(\cdot+k_{0})|_{(0,1)}$, where $k_{0}$ is the argument such that the above maximum is attained. Clearly, $\int_0^1 \mu (\|\tilde{u}(s)\|_{U}) \, ds \leq \int_0^t \mu (\|u(s)\|_{U}) \, ds$. We now use the assumptions on $Z$ described in the introduction. By c), $u(\cdot+k_{0})\in Z$ and $\|u(\cdot+k_{0})\|_{Z(0,t;U)}\leq \|u\|_{Z(0,t;U)}$. Therefore, Property d) implies that $\tilde{u}\in Z(0,1;U)$ with $\|\tilde{u}\|_{Z(0,1;U)}\leq \|u(\cdot+k_{0})\|_{Z(0,t;U)}\leq \|u\|_{Z(0,t;U)}$.
\end{proof}
Note that \ref{lem2:it1} in Lemma \ref{lem2} for the case $Z=L^p$ is well-known and can e.g.\ be found in \cite{TucWei09}
for $p=2$. 
}
\begin{proposition}\label{proprel1}
Let $Z\subseteq L_{loc}^{1}(0,\infty;U)$ be a function space. Then we have:
\begin{enumerate}[label=(\roman*)]
\item\label{proprel1:it1} The following statements are equivalent
\begin{enumerate}
\item $\Sigma(A,B)$ is $Z$-ISS,
\item  $\Sigma(A,B)$ is $Z$-admissible and $(T(t))_{t\ge 0}$  is exponentially stable, 
\item   $\Sigma(A,B)$ is $Z$-infinite-time admissible and $(T(t))_{t\ge 0}$  is exponentially stable.
\end{enumerate}
\item \label{proprel1:it2} If $\Sigma(A,B)$ is  $Z$-iISS,  then the system is $Z$-admissible and $(T(t))_{t\ge 0}$  is exponentially stable, 
\item \label{proprel1:it3}If $\Sigma(A,B)$ is $Z$-UBEBS, then  the system is $Z$-admissible and $(T(t))_{t\ge 0}$  is bounded, that is, \eqref{eqnexp} holds for $\omega=0$.
\end{enumerate}
\end{proposition}
\begin{proof} Clearly, $Z$-ISS, $Z$-iISS and $Z$-UBEBS imply $Z$-admissibility. Further, $Z$-ad\-missi\-bility and exponential stability of $(T(t))_{t\ge 0}$ show $Z$-ISS, see Remark \ref{remadm}. If, $\Sigma(A,B)$ is $Z$-ISS or $Z$-iISS, by setting $u=0$, it follows that $\|T(t)\|<1$ for sufficiently large $t$, which shows that $(T(t))_{t\ge 0}$  is exponentially stable. It is easy to see that $Z$-UBEBS implies boundedness of $(T(t))_{t\ge 0}$. Finally, by
Remark \ref{remadm}  items (b) and (c) in (i) are equivalent.
\end{proof}
\begin{proposition}\label{proprel2}
If $1\le p <\infty$, then the following are equivalent
\begin{enumerate}[label=(\roman*)]
\item\label{proprel2i} $\Sigma(A,B)$ is $L^p$-ISS,
\item\label{proprel2ii}  $\Sigma(A,B)$ is $L^p$-iISS,
\item\label{proprel2iii} $\Sigma(A,B)$ is $L^p$-UBEBS and $(T(t))_{t\ge 0}$  is exponentially stable.
\end{enumerate}
\end{proposition}
\begin{proof}
Clearly, by the definition of iISS and UBEBS, \ref{proprel2ii}$\Rightarrow$ \ref{proprel2iii}. By Proposition \ref{proprel1}, \ref{proprel2iii}$ \Rightarrow $\ref{proprel2i}. Thus in view of Proposition \ref{proprel1} it remains to show that $L^p$-infinite-time admissibility and  exponential stability imply $L^p$-iISS. Indeed, $L^p$-infinite-time admissibility and  exponential stability show for $x_0\in X$ and $u\in L^p(0,t;U)$ that
\begin{align*}
\Vert x(t) &\Vert \leq M {\rm{e}}^{-\omega t}\|x_0\|+ c_\infty \left \Vert u \right \Vert_{L^p(0,t; U)}\\ &= M {\rm{e}}^{-\omega t}\|x_0\|+c_\infty \left(\int_0^t \|u(s)\|_U^p \, ds \right)^{1/p},
\end{align*}
which shows $L^p$-iISS.
\end{proof}

\begin{rem}
Let $1\le p<\infty$. If  the system $\Sigma(A,B)$ is  $L^p$-admissible  and $(T(t))_{t\ge 0}$ is exponentially stable, then 
the system $\Sigma(A,B)$ is $L^p$-ISS with the following choices for the functions $\beta$ and $\mu$:
\[ \beta(s,t):= M{\rm{e}}^{-\omega t}s \quad \mbox{and} \quad \mu(s):= c_\infty s.\]
Here the constants $M$ and $\omega$ are given by \eqref{eqnexp} and $c_\infty =\sup_{t\ge 0} c(t)$.
\end{rem}

\begin{proposition}\label{prop:zeroclass}
If $\Sigma(A,B)$ is $L^\infty$-iISS, then $\Sigma(A,B)$ is $L^\infty$-zero-class admissible.
\end{proposition}
\begin{proof}
If $\Sigma(A,B)$ is $L^\infty$-iISS, then there exist $\theta \in \mathcal K_\infty$ and $\mu \in \mathcal K$ such that for all $t >0$, $u \in L^\infty(0,t;U)$, $u \neq 0$
 \begin{equation}
   \label{eq:1}
  \frac{1}{\|u\|_{\infty}}\left\|\int_0^t T_{-1}(s) B u(s)\,  ds \right \| \leq \theta \left(\int_0^t \mu \left(\tfrac{\|u(s)\|_U}{\|u\|_\infty}\right) \, ds\right).
 \end{equation}
Since the function $\mu$ is monotonically increasing and $\|u(s)\|_U \leq \|u\|_\infty$ a.e., the right-hand side of \eqref{eq:1} is bounded above by $\theta(t\mu(1))$ which converges to zero as $t \searrow 0$.
\end{proof}

 We illustrate the relations of the different stability notions with respect to  $L^\infty$ discussed above in the diagram depicted in Figure \ref{fig1}.
 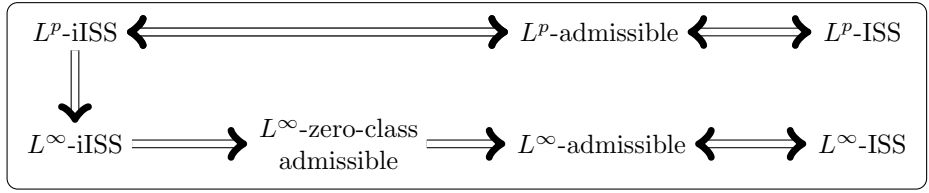
\begin{figure}\center
 \begin{tikzpicture}[->,
 shorten >=1pt,auto,node distance=3.5cm, double distance=2.5pt,framed,rounded corners]

   \node (A)                    {$L^{p}$-iISS};
   \node (B) [right=5.1cm of A] {$L^{p}$-admissible};
   \node (C) [right of =B] {$L^{p}$-ISS};
    \node (D) [below=1cm of A] {$L^{\infty}$-iISS};
    \node (E) [align=center,right of=D] {$L^{\infty}$-zero-class\\ admissible};
     \node (G) [right of=E] {$L^{\infty}$-admissible};
    \node (F) [below=1cm of C] {$L^{\infty}$-ISS};
       \path (B) edge [<->,double]           (C);
     \path (A) edge [<->,double]    (B);
	 \path (D) edge [double]           (E);
	  \path (G) edge [<->,double]           (F);
     \path (A) edge [<->,double]    (B);
       \path (E) edge [double]         (G);
       
       	  \path (A) edge [double]           (D);
 \end{tikzpicture}
 \caption{Relations between the different stability notions  with respect to  $L^p$, $p<\infty$, and $L^{\infty}$ for a system $\Sigma(A,B)$, where it is assumed that the semigroup is exponentially stable.\newline
 }
 \label{fig1}
 \end{figure}

\begin{proposition} \label{prop1b}  Suppose that $B$ is a bounded operator from $U$ to $X$ and $Z\subseteq L_{loc}^{1}(0,\infty;U)$ is a function space as in Section \ref{sec:2.1}. Then the following statements are equivalent.
\begin{enumerate}[label=(\roman*)]
\item\label{prop:it1} $(T(t))_{t\ge 0}$  is exponentially stable,
\item\label{prop:it2}  $\Sigma(A,B)$ is $Z$-admissible  and $(T(t))_{t\ge 0}$  is exponentially stable,
\item\label{prop:it3}  $\Sigma(A,B)$ is $Z$-infinite-time admissible and $(T(t))_{t\ge 0}$  is exponentially stable,
\item\label{prop:it4}  $\Sigma(A,B)$ is $Z$-ISS,

\item\label{prop:it5}  $\Sigma(A,B)$ is $Z$-iISS,
\item \label{prop:it6} $\Sigma(A,B)$ is $Z$-UBEBS and $(T(t))_{t\ge 0}$  is exponentially stable,
\item\label{prop:it7}  $\Sigma(A,B)$ is $L^{1}_{loc}$-admissible   and $(T(t))_{t\ge 0}$  is exponentially stable.
\end{enumerate}
If $Z$ satisfies Assumption (B), then the above assertions are equivalent to
\begin{enumerate}
\item[{\it (viii)}]  $\Sigma(A,B)$ is $Z$-zero-class admissible  and $(T(t))_{t\ge 0}$  is exponentially stable.
\end{enumerate}
\end{proposition}

\begin{proof}
By Proposition \ref{proprel1} we have $\ref{prop:it5}\Rightarrow \ref{prop:it6} \Rightarrow \ref{prop:it2} \Rightarrow \ref{prop:it3} \Rightarrow \ref{prop:it4} \Rightarrow \ref{prop:it1}$, and Proposition \ref{proprel2} and Remark \ref{remISS} \label{remISSit1} prove $  \ref{prop:it7} \Rightarrow \ref{prop:it5}$. The implication $\ref{prop:it1} \Rightarrow \ref{prop:it7}$ follows from the fact that 
by the boundedness of $B$ we have $x(t)\in X$ for all $t\geq0$ and all $u\in L^1(0,t;U)$.
Clearly, $(viii) \Rightarrow \ref{prop:it2}$. Thus it remains to show that if $Z$ satisfies Assumption (B), then $\ref{prop:it1} \Rightarrow (viii)$. 
Let $(T(t))_{t\ge 0}$ be exponentially stable, that is, there exist constants $M,\omega>0$ such that \eqref{eqnexp} holds. Therefore, for any $u\in L^{1}(0,t;U)$,
\begin{align}
\|x(t)\| &\le M{\rm{e}}^{-\omega t}\|x_0\| + M \|B\| \int_0^t {\rm{e}}^{-\omega(t- s)}\|u(s)\|_{U}\, ds\notag \\
&\le  M{\rm{e}}^{-\omega t}\|x_0\| +  M\|B\| \int_{0}^{t}\|u(s)\|_{U}\, ds\label{prop1:eq1}.
\end{align}
Using  that $Z(0,t;U)$ is continuously embedded in $L^{1}(0,t;U)$, we conclude that 
\begin{equation}\label{prop1:eq2}\|x(t)\| \leq M{\rm{e}}^{-\omega t}\|x_0\| +  M\|B\| \kappa(t) \|u\|_{Z(0,t;U)}\end{equation}
for all $t\geq0$. If Assumption (B) holds, then the embedding constants $\kappa(t)$ tend to $0$ as $t\searrow 0$. Hence, \eqref{prop1:eq2} shows that \ref{prop:it1} implies (viii).
\end{proof}
For the special case $Z=L^{p}(0,\infty;U)$, parts of the equivalences in Proposition \ref{prop1b} can already be found in \cite{MiI16}.
\begin{rem}
Note that in Proposition \ref{prop1b}, the assertions are independent of $Z$ as the assertions only rest on  exponential stability. In particular, if one of the equivalent conditions hold, then the system $\Sigma(A,B)$ is  $L^{p}$-ISS with the following choices for the functions $\beta$ and $\mu$
\[ \beta(s,t):= M{\rm{e}}^{-\omega t}s \quad \mbox{and} \quad \mu(s):= \frac{M}{\omega q} \|B\| s,\]
where $q$ is the H\"{o}lder conjugate of $p$, and $L^p$-iISS with 
\[ \beta(s,t):= M{\rm{e}}^{-\omega t}s, \quad  \quad \mu(s):= s, \quad \mbox{and} \quad \theta(s) := sM\|B\|.\]
Here the constants $M$ and $\omega$ are given by \eqref{eqnexp}. Although, in this case a system is $L^p$-ISS or $L^p$-iISS for all $p$ if this holds for some $p$, the choices for the functions $\mu$, however, do depend on $p$. Note that if $B$ is unbounded, then the question whether a system is $L^{p}$-ISS or $L^{p}$-iISS crucially depends on $p$. \\
Furthermore, note that in the trivial case $X=U=\C$ and $A=-1$, $B=1$, we have that the system $\Sigma(A,B)$ is not $L^{1}$-zero-class admissible.
\end{rem}

\section{IISS from the viewpoint of Orlicz spaces}
\label{sec:Orlicz}

In this section we relate $L^\infty$-ISS and $L^1$-ISS to ISS with respect to Orlicz spaces $E_\Phi$ corresponding to a Young function $\Phi$. The use of Orlicz spaces is motivated by the idea of understanding the integral appearing in the definition of iISS, \eqref{intro:eq2}, 
 as some type of norm. For the definition and fundamental properties of Orlicz spaces and Young functions, we refer to the Appendix.
The main results of this section are summarized in the following three theorems.

\begin{theorem}\label{thm:main}
  The following statements are equivalent.
 \begin{enumerate}[label=(\roman*)]
 \item\label{thm:main:it1}  There is a Young function $\Phi$ such that the system $\Sigma(A,B)$ is $E_\Phi$-ISS,
 \item\label{thm:main:it2} $\Sigma(A,B)$ is $L^\infty$-iISS,
 \item\label{thm:main:it3} $(T(t))_{t\ge 0}$ is exponentially stable and there is a Young function $\Phi$ such that the system $\Sigma(A,B)$ is $E_\Phi$-UBEBS.
\end{enumerate}
\end{theorem}
If $\Phi$ satisfies the $\Delta_{2}$-condition (see Definition \ref{def:deltatwo}) more can be said.
\begin{theorem}\label{thm:genLp} If $\Phi$ is a Young function that satisfies the $\Delta_{2}$-condition, then the following are equivalent.
\begin{enumerate}[label=(\roman*)]
\item\label{thm:genLp1} $\Sigma(A,B)$ is $E_{\Phi}$-ISS,
\item\label{thm:genLp2}   $\Sigma(A,B)$ is $E_{\Phi}$-iISS,
\item\label{thm:genLp3}  $\Sigma(A,B)$ is $E_{\Phi}$-UBEBS and $(T(t))_{t\ge 0}$  is exponentially stable.
\end{enumerate}

\end{theorem}
\begin{rem}
Since $L^{p}$-spaces are examples of Orlicz spaces where the $\Delta_{2}$-condition is satisfied, Theorem \ref{thm:genLp} can be seen as a generalization of Proposition \ref{proprel2}.
\end{rem}

\begin{theorem}\label{thm:main2}
  Then the following statements are equivalent.
  \begin{enumerate}[label=(\roman*)]
  \item\label{thm2:it2}  $\Sigma(A,B)$ is $L^1$-ISS,
  \item \label{thm2:it3} $\Sigma(A,B)$ is $L^1$-iISS,
  \item \label{thm2:it4}  $\Sigma(A,B)$ is $E_\Phi$-ISS for every Young function $\Phi$.
  \end{enumerate}
\end{theorem}

The proofs of Theorems \ref{thm:main}, \ref{thm:genLp} and  \ref{thm:main2} are given at the end of this section.

 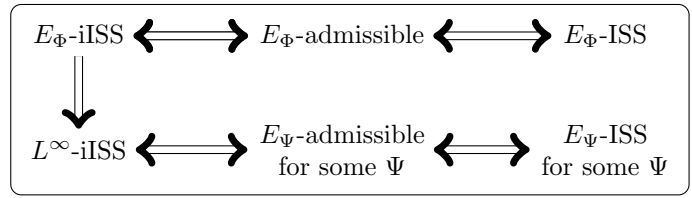
\begin{figure}[h]\center
 \begin{tikzpicture}[->,
 shorten >=1pt,auto,node distance=3.5cm, double distance=2.5pt,
framed,rounded corners]

   \node (A)                    {$E_{\Phi}$-iISS};
   \node (B) [right of=A] {$E_{\Phi}$-admissible};
   \node (C) [right of =B] {$E_{\Phi}$-ISS};
    
    \node (D) [below=1cm of A] {$L^{\infty}$-iISS};
     \node (E) [right of=D,align=center] {$E_{\Psi}$-admissible\\ for some $\Psi$};
    \node (F)  [align=center,right of=E] {$E_{\Psi}$-ISS\\ for some $\Psi$};
       \path (B) edge [<->,double]           (C);
     \path (A) edge [<->,double]    (B);
	 \path (D) edge [<->,double]           (E);
	  \path (E) edge [<->,double]           (F);
     \path (A) edge [<->,double]    (B);
       
       	  \path (A) edge [double]           (D);
 \end{tikzpicture}
 \caption{Relations between the different stability notions  with respect to Orlicz spaces for a system $\Sigma(A,B)$, where it is assumed that the semigroup is exponentially stable and that $\Phi$ satisfies the $\Delta_{2}$-condition.
 }
 \label{fig3}
 \end{figure}

\begin{lemma}\label{lem:3.24}
Let $\Sigma(A,B)$ be $L^\infty$-iISS. Then there exist $\tilde\theta, \Phi \in \mathcal K_\infty$ such that $\Phi$ is a Young function which is continuously differentiable on $(0,\infty)$ and
  \begin{equation} \label{eq:4}
       \left \Vert \int_0^t T_{-1} (s) B u(s) \, ds \right \Vert \leq \tilde\theta \left(\int_0^t \Phi (\Vert u(s) \Vert_U )\, ds \right)
  \end{equation}
for all $t>0$ and $u \in L^\infty (0,t;U)$. 
\end{lemma}

{
\begin{proof}
  By assumption, $(T(t))_{t\ge 0}$ is exponentially stable and there exist $\theta \in \mathcal K_\infty$ and $\mu \in \mathcal K$ such that (11) holds for $Z = L^\infty$. Without loss of generality we can assume that $\mu$ belongs to $\mathcal K_\infty$. By Lemma 14 in \cite{PraWa96} there exist a convex function $\mu_v\in \mathcal K_\infty$ and a concave function $\mu_c \in \mathcal K_\infty$ such that both are continuously differentiable on $(0, \infty)$ and $\mu \leq \mu_c \circ \mu_v$ holds on $[0, \infty)$. Now for any Young function $\Psi \colon [0,\infty) \to [0, \infty)$ it is straight forward to check that $\mu_c \circ \Psi^{-1}$ is a concave function and hence we have by Jensen's inequality
 \begin{equation*}
    \begin{split}
      \theta \left(\int_0^1 \mu (\Vert u(s) \Vert_U )\, ds \right) &\leq \theta \left(\int_0^1  \mu_c \circ \mu_v (\Vert u(s) \Vert_U )\, ds \right) \\ & \leq (\theta\circ \mu_c  \circ \Psi^{-1})\left(\int_0^1 (\Psi \circ \mu_v)(\Vert u(s) \Vert_U) \,ds  \right).
    \end{split}
  \end{equation*}
Using Remark 3.2.7 in \cite{Kufner} it is easy to see that  $\Phi:= \Psi \circ \mu_v$ is a Young function. Taking $\tilde \theta := \theta\circ \mu_c  \circ  \Psi^{-1}$ we obtain the desired estimate for $t= 1$. By Lemma \ref{lem2}, the assertion follows.
\end{proof}
}
\begin{proof}[Proof of Theorem  \ref{thm:main}]
\ref{thm:main:it1} $\Rightarrow$ \ref{thm:main:it2}: Since $\Lambda(s)=s^2$ defines a Young function with $\Lambda(1) = 1$, it can be easily seen that
\begin{equation*}
  \Phi_1(s) =
  \begin{cases}
    \Phi(s), &s< 1, \\
    \Phi(\Lambda(s)), &s\geq 1,
  \end{cases}
\end{equation*}
defines another Young function such that $\Phi \leq \Phi_1$. Furthermore, $\Phi_{1}$  increases essentially more rapidly than $\Phi$ (see Def.\ \ref{def:increase}), since the composition $\Phi \circ \Lambda$ of two Young functions $\Phi,\Lambda$ is known to be increasing essentially more rapidly than $\Phi$ (see page 114 of \cite{KrasnRut}).
We define $\theta \colon [0,\infty) \to [0,\infty)$ by
\begin{equation*}
 \theta (\alpha) = \sup\left\{ \left\Vert \int_0^1 T_{-1}(s) B u(s) \,ds  \right\Vert \Bigm| u \in L^\infty (0,1;U),\int_0^1 \Phi_{1}(\|u(s)\|_U) \,ds \leq \alpha \right\},
\end{equation*} 
for $\alpha>0$ and $\theta(0)=0$. Clearly, $\theta$ is non-decreasing.
Admissibility with respect to $E_\Phi$ and Remark \hyperref[rem:orlicz-part4]{\ref*{rem:orlicz}.\ref*{rem:orlicz-part4}}  yield that for $u\in L^{\infty}(0,1;U)$,
\[\left\|\int_{0}^{1}T_{-1}(s)Bu(s)\,ds\right\|\leq c(1)\|u\|_{E_{\Phi}(0,1;U)} \leq c(1)\left(1+\int_{0}^{1}\Phi_{1}(\|u(s)\|_U)\,ds\right).\]
Hence,  $\theta(\alpha)<\infty$ for all $\alpha\geq0$. \newline
If we can show that $\lim_{t \searrow 0}\theta(t)=0$, then, by Lemma 2.5 in \cite{CLS98}, 
there exists  $\tilde{\theta}\in \mathcal{K}_\infty$ such that $\theta\leq\tilde{\theta}$ pointwise. Therefore, let $(\alpha_{n})_{n \in \mathbb N}$ be a sequence of positive real numbers converging to $0$. By the definition of $\theta$, for any $n\in\N$ there exists $u_{n}\in L^{\infty}(0,1;U)$ such that 
\begin{equation*}
  \int_0^{1} \Phi_1(\|u_n(s)\|_U) \,ds \le\alpha_n
\end{equation*}
and 
\begin{equation}\label{thm31:eq5}
  \left|\theta(\alpha_{n})- \left \Vert \int_{0}^{1}T_{-1}(s)Bu_{n}(s) \,ds\right \Vert \right|<\frac{1}{n}.\end{equation}
Hence the sequence $(\|{u_{n}(\cdot)}\|_{U})_{n \in \N}$ is $\Phi_{1}$-mean convergent to zero (see Def.\ \ref{def:meanconv}). 
By Theorem \ref{thm:increasemorerapidly}, the sequence even converges to zero with respect to the norm of the space $L_{\Phi}(0,1)$, and thus also in $E_{\Phi}(0,1)$. Hence
\[\lim_{n \to \infty} \|{u}_{n}\|_{E_{\Phi}(0,1;U)}=\lim_{n \to \infty}\left\|\| u_n(\cdot)\|_{U}\right\|_{E_{\Phi}(0,1)} = 0,\]
 where we used Remark \hyperref[rem:orlicz-part2]{\ref*{rem:orlicz}.\ref*{rem:orlicz-part2}}. Hence, by admissibility,
\[  \left \Vert \int_0^{1} T_{-1}(s) B  u_{n}(s) \, ds \right \Vert \leq c(1) \| u_n\|_{E_\Phi(0,1;U)} \to 0, \] 
as $n\to \infty$. Altogether we obtain that
\begin{equation*}
  \begin{split}
    \theta(\alpha_n) &\leq \left| \theta(\alpha_n) - \left \Vert\int_{0}^{1}T_{-1}(s)Bu_{n}(s) \,ds \right \Vert \right| + \left \Vert\int_{0}^{1}T_{-1}(s)Bu_{n}(s) \,ds \right \Vert\\ &\leq \frac{1}{n} + c(1) \| u_n\|_{E_\Phi(0,1;U)}, 
  \end{split}
\end{equation*}
and thus, $\lim_{n \to \infty} \theta(\alpha_n) =0$.

 Therefore, there exists $\tilde{\theta}\in\mathcal{K}_{\infty}$ such that $\theta\leq\tilde{\theta}$ pointwise. Furthermore, $\Phi_{1}$ is a Young function, in particular we have $\Phi_{1} \in \mathcal K_{\infty}$. The definition of $\theta$ yields that
 \[\left\Vert \int_0^1 T_{-1}(s) B u(s) \,ds  \right\Vert \leq \theta\left(\int_0^1 \Phi_{1}(\|u(s)\|_U) \,ds \right) \leq\tilde{\theta}\left(\int_0^1 \Phi_{1}(\|u(s)\|_U) \,ds \right)\]
 for all  $u \in L^\infty(0, 1;U)$. By Lemma \ref{lem2}, we conclude that $\Sigma(A,B)$ is $L^\infty$-iISS.\\

 \ref{thm:main:it2} $\Rightarrow$ \ref{thm:main:it1}: Now assume that $\Sigma(A,B)$ is $L^\infty$-iISS. 
  We need to show that for some Young function $\Phi$ the system $\Sigma(A,B)$ is $E_\Phi$-ISS. By Proposition  \ref{proprel1}(i) it suffices to show that there is a Young function $\Phi$ such that $\int_0^t T_{-1}(s) B u(s) \, ds \in X$ for all $u \in E_\Phi(0,t)$. Note that since $E_\Phi(0,t;U) \subset L^1(0,t;U)$ for any Young function $\Phi$, the integral always exists in $X_{-1}$. By assumption, $\int_0^t T_{-1}(s) B u(s)\ ds \in X$ for all $u \in L^\infty(0,t)$. By Lemma \ref{lem:3.24}, there exist $\tilde{\theta} \in \mathcal K_\infty$ and a Young function $\Phi$ such that (\ref{eq:4}) holds. Let $u \in E_\Phi$. By definition, there is a sequence $(u_n)_{n \in \mathbb N} \subset L^\infty(0,t;U)$ such that $\lim_{n \to \infty} \|u_n -u\|_{E_{\Phi}(0,t;U)} = 0$. Since $(u_n)_{n \in \mathbb N}$ is a Cauchy sequence in $E_\Phi(0,t;U)$, we can assume without loss of generality that $\|u_n -u_m\|_{E_{\Phi}(0,t;U)} \leq 1$ for all $m,n \in \mathbb N$. By \cite[Lemma 3.8.4 (i)]{Kufner} 
   this implies that for all $n,m\in\N$,
  \begin{equation*}
  	\int_0^t \Phi (\|u_n(s)-u_m(s)\|_{U}) \,ds \leq \|u_n -u_m\|_{E_{\Phi}(0,t;U)}.
  \end{equation*}
  Together with \eqref{eq:4} and the monotonicity of $\tilde{\theta}$, this yields
\begin{equation*}
  \begin{split}
    \left\Vert \int_0^t T_{-1}(s) B (u_n(s)-u_m(s)) \,ds  \right\Vert &\leq \tilde{\theta} \left(\int_0^t \Phi (\|u_n(s)-u_m(s)\|_{U}) \,ds \right)\\ 
     &\leq \tilde{\theta}\left(\|u_n -u_m\|_{E_{\Phi}(0,t;U)}\right).
  \end{split}
\end{equation*}
Hence $(\int_0^t T_{-1}(s) B u_n(s) \,ds)_{n \in \mathbb N}$ is a Cauchy sequence in $X$ and thus converges. Let $y$ denote its limit. 
Since $E_{\Phi}(0,t;U)$ is continuously embedded in $L^{1}(0,t;U)$, see Remark  \hyperref[rem:orlicz-part3]{\ref*{rem:orlicz}.\ref*{rem:orlicz-part3}}, it follows  that
\begin{equation*}
  \lim_{n \to \infty} \int_0^t T_{-1}(s) B u_n(s) \,ds =\int_0^t T_{-1}(s) B u(s) \,ds
\end{equation*}
in $X_{-1}$. Since $X$ is continuously embedded in $X_{-1}$, we conclude that \[y = \int_0^t T_{-1}(s) B u(s) \,ds.\]
 Thus, we have shown that $\int_0^t T_{-1}(s) B u(s) \,ds \in X$ for all $u \in E_\Phi$ and hence $\Sigma(A,B)$ is admissible with respect to $E_\Phi$.\\
 
 \ref{thm:main:it1} $\Rightarrow$ \ref{thm:main:it3}: This follows since for all $u\in E_{\Phi}(0,t;U)$ it holds that $u\in \tilde{L}_{\Phi}(0,t;U)$ and 
 \[\|u\|_{E_{\Phi}} \leq 1+\int_{0}^{t}\Phi(\|u(s)\|_{U})\, ds,\]
 see Remark \hyperref[rem:orlicz-part4]{\ref*{rem:orlicz}.\ref*{rem:orlicz-part4}}. 

\ref{thm:main:it3} $\Rightarrow$ \ref{thm:main:it1}: This follows by \ref{proprel1:it3} and \ref{proprel1:it1} of Proposition \ref{proprel1}.
\end{proof}

\begin{proof}[Proof of Theorem \ref{thm:genLp}]
The implications \ref{thm:genLp2} $\Rightarrow$ \ref{thm:genLp3} $\Rightarrow$ \ref{thm:genLp1} follow, analogously as for the $L^p$-case, by Proposition \ref{proprel1}.
 
 \ref{thm:genLp1} $\Rightarrow$ \ref{thm:genLp2}: Similarly to the proof of Theorem \ref{thm:main}, we can define a non-decreasing function $\theta$ by
 \[\theta(\alpha)=\sup\left\{\left\|\int_{0}^{1}T_{-1}(s)Bu(s)\,ds\right\| \Bigm| u \in E_{\Phi}(0,1;U), \int_{0}^{1}\Phi(\|u(s)\|_U)\,ds\leq \alpha\right\},\]
for $\alpha>0$ and $\theta(0):=0$. By $E_\Phi$-admissibility and Remark \hyperref[rem:orlicz-part4]{\ref*{rem:orlicz}.\ref*{rem:orlicz-part4}}, we have that 
\[\left\|\int_{0}^{1}T_{-1}(s)Bu(s)\,ds\right\|\leq c(1)\|u\|_{E_{\Phi}(0,1;U)} \leq c(1)\left(1+\int_{0}^{1}\Phi(\|u(s)\|_U)\,ds\right),\]
for $u\in E_{\Phi}(0,1;U)\subset \tilde{L}_{\Phi}(0,t;U)$. Hence, $\theta$ is well-defined. In analogy to the proof of Theorem \ref{thm:main}, it remains to show that $\theta$ is right-continuous at $0$. This follows because $\Phi$ satisfies the $\Delta_{2}$-condition. In fact, if the latter is true, it is known that a sequence $(u_{n})_{n\in\N}$ in $E_{\Phi}$ converges to $0$ if and only if the sequence is $\Phi$-mean convergent to zero (see Def.~\ref{def:meanconv}). Therefore, $\alpha_{n}\searrow 0$ implies that there exists a sequence $u_{n}\in E_{\Phi}(0,1;U)$ that converges to $0$ in $E_{\Phi}$ and such that 
\[\left|\theta(\alpha_{n})-\left\|\int_{0}^{1}T_{-1}Bu_{n}(s)\,ds\right\|\right|\leq \frac{1}{n},\quad n\in\N.\]
By $E_{\Phi}$-admissibility, we conclude that $\theta(\alpha_{n})\to0$ as $n\to\infty$.\newline
Hence, by Lemma 2.4 in \cite{CLS98}, we find $\tilde{\theta}\in\mathcal{K}_{\infty}$ such that $\theta\leq\tilde{\theta}$ pointwise. By definition of $\theta$, this implies
\[\left\|\int_{0}^{1}T_{-1}(s)Bu(s)\,ds\right\| \leq \tilde{\theta}\left(\int_{0}^{1}\Phi(\|u(s)\|_U)\,ds\right)\]
for all $u\in E_{\Phi}(0,1;U)$. Finally, Lemma \ref{lem2} yields that $\Sigma(A,B)$ is $E_{\Phi}$-iISS.
\end{proof}

\begin{proof}[Proof of Theorem \ref{thm:main2}]
By Propositions \ref{proprel1} and \ref{proprel2}, we only need to show the equivalence of \ref{thm2:it2} and  \ref{thm2:it4}. That  \ref{thm2:it2}\ implies  \ref{thm2:it4} follows immediately since $E_\Phi$ is continuously embedded in $L^1$. \\
Conversely, let $\Sigma(A,B)$ be $E_\Phi$-admissible  for every Young function $\Phi$. According to Proposition \ref{proprel1} (a), we have to show that $\Sigma(A,B)$ is $L^{1}$-admissible. Let  $t >0$ and $u \in L^1(0,t;U)$. It remains to prove that $\int_0^t T_{-1}(s) Bu(s) \,ds \in X$. 
By \cite[p. 61]{KrasnRut}, there exists a Young function $\Phi$ satisfying the $\Delta_2$-condition such that $\|u(\cdot)\|_U \in L_\Phi$\footnote{In \cite[p.~61]{KrasnRut} it is actually shown that for given $f\in L^{1}(0,t)$, there exists a Young function $Q$ such that $f\in L_{Q\circ Q}(0,t)$ and such that $Q$  satisfies the $\Delta'$-condition, i.e., 
\[\exists c,u_{0}>0\ \forall u,v\ge u_{0}:\quad Q(uv)\le cQ(u)Q(v).\] 
 In fact, it is easy to see that this property  implies that $Q\circ Q$ satisfies
 \[\forall u\ge u_{0}:\quad(Q\circ Q)(\ell u)\leq k(\ell) (Q\circ Q)(u),\]
 for some $\ell>1$ and $k(\ell)>0$, which is known to be equivalent to $Q\circ Q$ satisfying the $\Delta_{2}$-condition, see \cite[p.~23]{KrasnRut}. }.  
The $\Delta_2$-condition implies that $E_{\Phi}=L_{\Phi}$ and $E_\Phi(0,t;U) = L_\Phi(0,t;U)$, see \cite[p.~303]{RaoOrlicz} or \cite[Thm.~5.2]{Schappacher}.  Thus $\int_0^t T_{-1}(s) Bu(s) \,ds \in X$ by assumption. 
\end{proof}

{
\begin{proposition}\label{prop:iisssuff}
 Let $\Sigma(A,B)$ be $L^\infty$-ISS. If there exist a nonnegative function  $f\in L^1(0,1)$, $\theta\in{\calK} $, a constant $c>0$ and  a Young function $\mu$  such that for every  $u\in L^1(0,1;U)$ with $ \int_0^1 f(s) \mu (\|u(s)\|_{U}) \, ds<\infty$ one has
\begin{equation*}
\left\|   \int_0^1 T_{-1} (s) B u(s) \, ds  \right\| \le c+ \theta\left( \int_0^1 f(s) \mu (\|u(s)\|_{U}) \, ds\right) ,\end{equation*}
then  $\Sigma(A,B)$ is $L^\infty$-iISS.
\end{proposition}

\begin{proof} By Theorem \ref{thm:main} and Proposition \ref{proprel1} it is sufficient to show that  there is a Young function $\Phi$ such that the system $\Sigma(A,B)$ is $E_\Phi$-admissible. Theorem \ref{theo:l1orlicz} implies that  there exists a Young function $\Psi$ such that $f\in \tilde L_{\Psi}(0,1)$. Let $\tilde \Phi$ be the complementary Young function to $\Psi$. We define the Young function $\Phi$ by $\Phi:= \tilde\Phi\circ \mu$. Using 
Remark \ref{rem:A:rhoOrlicz} for  $u\in E_\Phi(0,1;U)$  we obtain 
\begin{align*}
\left\|   \int_0^1 T_{-1} (s) B u(s) \, ds  \right\| \le &\, \, c+  \theta\left(\int_0^1 f(s) \mu (\|u(s)\|_{U}) \, ds \right)\\
\le & \,\,c+  \theta\left( \int_0^1 \Psi(f(s))\, ds + \int_0^1 \tilde\Phi(\mu(\|u(s)\|_U)\, ds\right).
\end{align*}
This shows that for all $u \in E_\Phi(U)$ we have 
\begin{equation*} 
        \int_0^1 T_{-1} (s) B u(s) \, ds   \in X,
  \end{equation*}
 that is, $\Sigma(A,B)$ is $E_\Phi$-admissible.
\end{proof}
}

\section{Stability of parabolic diagonal systems} \label{sec:diag.syst}
In the previous section we have proved that for infinite-dimensional systems $L^\infty$-iISS implies $L^\infty$-ISS. It is an open question whether the converse implication holds. Here, we give a positive answer for parabolic diagonal systems, which are a well-studied class of systems in the literature, see e.g.\ \cite{TucWei09}.\\
Throughout this section  we assume that $U=\mathbb C$, $1\leq q < \infty$ and that the operator $A$ possesses a $q$-Riesz basis of eigenvectors $(e_n)_{n\in \mathbb N}$ with eigenvalues $(\lambda_n)_{n\in\mathbb N}$ lying in a sector in the open left half-plane $\mathbb C_-$. More precisely,  $(e_n)_{n\in \mathbb N}$ is a $q$-Riesz basis of $X$, if $(e_n)_{n\in \mathbb N}$ is a Schauder basis and for some constants $c_1,c_2>0$ we have
\[ c_1\sum_{k} |a_k|^q \le \left\| \sum_{k} a_k e_k \right\|^q\le  c_2 \sum_{k} |a_k|^q\]
for all sequences $(a_k)_{k\in\mathbb N}$ in $\ell^{q}=\ell^q(\N)$. Thus without loss of generality we can assume that $X=\ell^q$ and  that $(e_n)_{n\in \mathbb N}$ is the canonical basis of $\ell^q$. We further assume that the sequence $(\lambda_n)_{n\in\mathbb N}$ lies in $\mathbb C$ with $\sup_{n}\Re( \lambda_{n})<0$ and that there exists a constant $k>0$ such that $|\text{Im}\, \lambda_{n}|\le k|\text{Re}\, \lambda_{n}|$, $n\in \mathbb N$,  i.e., $(\lambda_{n})_{n}\subset \mathbb C \setminus S_{\pi/2+\theta}$ for some $\theta\in(0,\pi/2)$, where
\[S_{\pi /2 + \theta}=\{z\in\C \mid |z|>0, |\arg z|<\pi/2+\theta\}.\]
 Then the linear operator $A \colon D(A)\subset \ell^q\rightarrow \ell^q$, given by 
\[ Ae_n = \lambda_n e_n, \qquad n\in\mathbb N,\]
and $D(A)= \{(x_n)\in \ell^q \mid \sum_{n} |x_n \lambda_n|^q <\infty\}$, generates an analytic exponentially stable $C_0$-semigroup $(T(t))_{t\ge 0}$ on $\ell^q$, which is given by
$T(t)e_n= {\rm{e}}^{t\lambda_n}e_n$. An easy calculation shows that the extrapolation space $(\ell^q)_{-1}$ is given by 
\begin{align*}(\ell^q)_{-1}={}&\left\{x=(x_{n})_{n\in\N}\:|\: \sum_{n}\frac{|x_{n}|^{q}}{|\lambda_{n}|^{q}}<\infty\right\},\\ \|x\|_{X_{-1}}={}&\|A^{-1}x\|_{\ell^{q}}.\end{align*}
Thus the linear bounded operator $B$ from $\mathbb C$ to $(\ell^q)_{-1}$ can be identified with a sequence $(b_n)_{n\in \mathbb N}$  in $\mathbb C$ satisfying
\[ \sum_{n\in\N} \frac{|b_{n}|^{q}}{|\lambda_{n}|^{q}}<\infty.\]
Thanks to the sectoriality condition for $(\lambda_n)_{n\in\mathbb N}$ this equivalent to
\[ \sum_{n\in\N} \frac{|b_{n}|^{q}}{|\Re\lambda_{n}|^{q}}<\infty.\]

The following result shows that, under the above assumptions, the system $\Sigma(A,B)$ is $L^\infty$-iISS. Thus for this class of systems $L^\infty$-iISS is equivalent to $L^\infty$-ISS, and both notions are implied by  $B\in  (\ell^q)_{-1}$, that is, 
$\sum_{n} \frac{|b_{n}|^{q}}{|\lambda_{n}|^{q}}<\infty$.   The following theorem generalizes the main result in \cite{CDCJNPS16}, where the case $q=2$ is studied.

\begin{theorem}\label{theodiag}
Let $U=\mathbb C$, and assume that the operator $A$ possesses a $q$-Riesz basis of $X$ that consists  of eigenvectors $(e_n)_{n\in \mathbb N}$  with eigenvalues $(\lambda_n)_{n\in\mathbb N}$ lying in a sector in the open left half-plane $\mathbb C_-$ with  $\sup_{n}\Re( \lambda_{n})<0$ and $B\in {\calL}(\C,X_{-1})$.
Then the system  $\Sigma(A,B)$ is $L^\infty$-iISS, and hence also $L^\infty$-ISS and $L^\infty$-zero-class admissible.
\end{theorem}

\begin{rem}
In the situation of Theorem
 \ref{theodiag}, $\Sigma(A,B)$ is  $L^\infty$-iISS if and only if $\Sigma(A,B)$ is  $L^\infty$-ISS.
\end{rem}

\begin{proof}[Proof of Theorem \ref{theodiag}] Without loss of generality we may assume $X=\ell^q$ and  that $(e_n)_{n\in \mathbb N}$ is the canonical basis of $\ell^q$. Let $f \colon (0,\infty) \to [0,\infty)$ be defined by \[f(s) = \sum_{n \in \mathbb N} \frac{|b_n|^q}{|\Re \lambda_n|^{q-1}} {\rm{e}}^{\Re \lambda_n s}.\]
Then it is easy to see that $f$ belongs to $L^1(0,\infty)$.
{Now for $u\in L^1(0,1)$ with  $\int_0^1 f(s) |u(s)|^q \,ds < \infty$ we obtain (denoting by  $q'$ the H\"older conjugate of $q$)}
{\allowdisplaybreaks 
\begin{align*}
   \left\Vert \int_0^1 T_{-1}(s)B u(s) \,ds  \right\Vert_{\ell^q}^q &= \sum_{n \in \mathbb N} |b_n|^q \left \vert \int_0^1 {\rm{e}}^{\lambda_n s} u(s) \,ds\right \vert^q\\
 & \hspace*{-2.3cm} \leq \sum_{n \in \mathbb N} |b_n|^q \left ( \int_0^1 {\rm{e}}^{\Re \lambda_n s} |u(s)| \,ds\right)^q \\
 & \hspace*{-2.3cm}= \sum_{n \in \mathbb N} \frac{|b_n|^q}{(\Re \lambda_n)^q} \left ( \int_0^1 |\Re \lambda_n|{\rm{e}}^{\Re \lambda_n s} |u(s)| \,ds\right)^q \\
  & \hspace*{-2.3cm}\leq \sum_{n \in \mathbb N} \frac{|b_n|^q}{(\Re \lambda_n)^q} \left ( \int_0^1 |\Re \lambda_n|{\rm{e}}^{\Re \lambda_n s} |u(s)|^q \,ds\right) \left(\int_0^1 |\Re \lambda_n|{\rm{e}}^{\Re \lambda_n s}  \,ds\right)^{q/q'} \\ 
  & \hspace*{-2.3cm}\leq \sum_{n \in \mathbb N} \frac{|b_n|^q}{|\Re \lambda_n|^q} \left ( \int_0^1 |\Re \lambda_n|{\rm{e}}^{\Re \lambda_n s} |u(s)|^q \,ds\right) \\
   & \hspace*{-2.3cm}= \int_0^1  \sum_{n \in \mathbb N} \frac{|b_n|^q}{|\Re \lambda_n|^{q-1}} {\rm{e}}^{\Re \lambda_n s} |u(s)|^q \,ds \\
    & \hspace*{-2.3cm}=  \int_0^1  f(s) |u(s)|^q \,ds\\ 
  & \hspace*{-2.3cm} < \infty. 
\end{align*}}
{This shows that the system $\Sigma(A,B)$ is $L^\infty$-ISS and the claim now follows from Proposition \ref{prop:iisssuff}.}
\end{proof}
{
\begin{rem}\label{rem:delta2}
Theorem \ref{theodiag} states that $L^\infty$-admissibility implies $E_{\Phi}$-admissibility for some Young function $\Phi$ in the case of parabolic diagonal systems. A natural question is whether $\Phi$ can always be chosen such that the $\Delta_2$-condition is satisfied. Looking at the proof and having in mind that $L^{1}$ equals the union of all spaces $E_{\Psi}$ where $\Psi$ satisfies the $\Delta_{2}$-condition, this could be expected. However, the answer is negative, which can be seen as follows. For a Young function $\Phi$ satisfying the $\Delta_{2}$-condition there exist constants $x_0>0$ and $p\in\N\setminus\{1\}$ such that
\begin{equation*}\Phi(x) \leq x^{p}, \qquad x>x_{0},\end{equation*}
see \cite[p.~25]{KrasnRut}. This implies that $E_{\Phi}\supset L^{p}$, see e.g.~\cite[Sec.~3.17]{Kufner}.
However, there exists Young functions that do not satisfy the latter estimate, e.g., $\Phi(x)={\rm{e}}^{x-1}-x-{\rm{e}}^{-1}$.
In Example \ref{ex27}, $\Sigma(A,B)$ is not $L^{p}$-admissible for any $p<\infty$, which, with the above reasoning, implies that the system cannot be $E_{\Phi}$-admissible for any $\Phi$ satisfying the $\Delta_2$-condition.
  \end{rem}
}

\begin{lemma}\label{lem:bddlap}
Let $\mu$ be a  {positive regular Borel} measure supported on a sector $S_\phi$ with $\phi\in(0,\frac{\pi}{2})$,
and let $1 \leq q < \infty$. Then the following are equivalent:
\begin{enumerate}[label=(\roman*)]
\item\label{itbddlap1} The Laplace transform $\mathcal{L}\colon L^\infty(0,\infty) \to L^q(\Cp,\mu)$ is bounded,
\item\label{itbddlap2}  The function $s \mapsto 1/s$ lies in $L^q(\Cp,\mu)$.
\end{enumerate}
\end{lemma}
\begin{proof}
\ref{itbddlap1}  $\Rightarrow$ \ref{itbddlap2}: Taking $f(t)=1$ for $t \ge 0$ we have that $\mathcal{L} f(s)=1/s$ and the result follows.\\
\ref{itbddlap2} $\Rightarrow$ \ref{itbddlap1}: For $f \in L^\infty(0,\infty)$  and $s \in \Cp$ we have
\[
\left| \int_0^\infty f(t) {\rm{e}}^{-st} \, dt \right| \le \|f\|_\infty \int_0^\infty |{\rm{e}}^{-st}| \, dt \le \|f\|_\infty/(\Re s) \le M \|f\|_\infty /|s|,
\]
where $M$ is a constant depending only on $\phi$.
Now Condition \ref{itbddlap2} implies that $\mathcal{L}$ is bounded.
\end{proof}
\begin{theorem}\label{thm:diagGenAna}
Suppose $A$ possesses a $q$-Riesz basis of $X$ consisting of eigenvectors $(e_n)_{n\in \mathbb N}$  with eigenvalues $(\lambda_n)_{n\in\mathbb N}$ lying in a sector in the open left half-plane $\mathbb C_-$ and $B\in X_{-1}$.
Then the following assertions are equivalent.
 
 \begin{enumerate}[label=(\roman*)]
\item \label{itThm251}$B$ is infinite-time admissible for $L^{\infty}$,
\item  \label{itThm252}$\sup_{\lambda\in\Cp}\|(\lambda-A)^{-1}B\|<\infty,$
\item  \label{itThm253}The function $s \mapsto 1/s$ lies in $L^q(\Cp,\mu)$,
where $\mu$ is the measure $\sum |b_k|^2 \delta_{-\lambda_k}$.\end{enumerate} 
\end{theorem}
\begin{proof}
By \cite[Thm~2.1]{JPP14}, admissibility is equivalent to the
boundedness of the Laplace transform $\mathcal{L} \colon L^\infty(0,\infty) \to L^q(\Cp,\mu)$,
and hence  \ref{itThm251} and  \ref{itThm253} are equivalent by Lemma \ref{lem:bddlap}.
Note that
\[
\|(\lambda-A)^{-1}B\|^q = \sum_k \frac{|b_k|^q}{|\lambda-\lambda_k|^q}.\]
Now if  \ref{itThm252} holds, then \ref{itThm253} also holds, letting $\lambda \to 0$.
Conversely, if  \ref{itThm253} holds, then by sectoriality we have that
\[
\sum_{k} \frac{|b_k|^q}{|\Re \lambda_k|^q} < \infty,
\]
and hence $\sum_k {|b_k|^q}/{|\lambda-\lambda_k|^q}$ is bounded
independently of $\lambda \in \Cp$, that is,  \ref{itThm252} holds.
\end{proof}

\begin{rem}
Let $\mathfrak{b}_{p}(X)$ denote the set of $L^{p}$-admissible control operators from $\C$ to $X$ for a given $A$. By Theorem \ref{theodiag}, we have that $\mathfrak{b}_{\infty}(X)=X_{-1}$ for exponentially stable, parabolic diagonal systems. Using \cite[Theorem 6.9]{Weiss89i}, and the inclusion of the $L^{p}$-spaces, we obtain the following chain of inclusions for $X=\ell^{q}$ with $q>1$\footnote{here, $q=1$ is also allowed if $(T^{*}(t))_{t\ge 0}$ is strongly continuous.}
\begin{equation}
X=\mathfrak{b}_{1}(X)\subset \mathfrak{b}_{p}(X) \subset \mathfrak{b}_{\infty}(X)=X_{-1}.
\end{equation}
It is not so hard to show that the equality $\mathfrak{b}_{\infty}(X)=X_{-1}$ does not hold in general if the exponential stability  is dropped. In fact, a counterexample on $X=\ell^{2}$ with the standard basis is given by $\lambda_{n}=2^{n}$, $n\in\Z$, $b_{n}=2^{n}/n$ for $n>0$ and $b_{n}=2^{n}$ for $n<0$.
\end{rem}

The relations of the different stability notions with respect to  $L^\infty$ for parabolic diagonal systems are summarized in the diagram shown in Figure \ref{fig2}.

 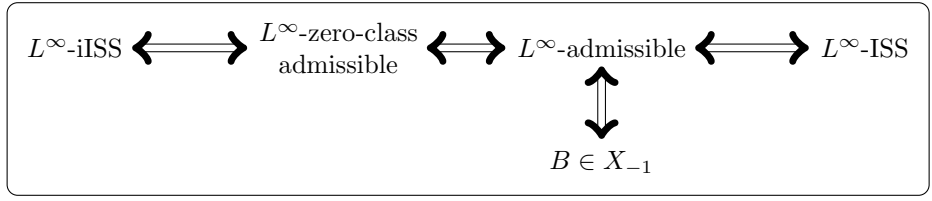
\begin{figure}\center
 \begin{tikzpicture}[->,shorten >=1pt,auto,node distance=3.5cm, double distance=2.5pt,framed,rounded corners]

    \node (D)  {$L^{\infty}$-iISS};
    \node (E) [align=center,right of=D] {$L^{\infty}$-zero-class\\ admissible};
     \node (G) [right of=E] {$L^{\infty}$-admissible};
    \node (F) [right of=G] {$L^{\infty}$-ISS};
    \node (H) [below=1cm of G] {$B\in X_{-1}$};
	 \path (D) edge [<->,double]           (E);
	  \path (G) edge [<->,double]           (F);
          \path (E) edge [<->,double]         (G);
       
       	  \path (G) edge [<->,double]           (H);
	 \end{tikzpicture}

\caption{Relations between the different stability notions for parabolic diagonal system (assuming that the semigroup is exponentially stable). 
 }
\label{fig2}
\end{figure}

\section{ Some Examples}\label{sec:ex}
\begin{example}\label{ex1}
Let us consider the following boundary control system given by the one-dimensional heat equation on the spatial domain $[0,1]$ with Dirichlet boundary control at the point $1$,
\begin{align*}
x_t(\xi,t)={}&a x_{\xi\xi}(\xi,t), \quad \xi\in(0,1),~ t>0,\\
  x(0,t)={}&0,  \quad x(1,t)=u(t), \quad t>0,\\
x(\xi,0)={}&x_{0}(\xi),
\end{align*}
where $a>0$.
It can be shown that this system can be written in the form $\Sigma(A,B)$ in \eqref{eqn1}. Here $X=L^{2}(0,1)$ and
\begin{align*}
Af={}&f'', \quad f\in D(A),\\
D(A)={}&\left\{f\in H^{2}(0,1) \mid  
f(0)= f(1)=0\right\}.
\end{align*}
Moreover, with $\lambda_{n}=-a\pi^{2}n^{2}$, \[Ae_{n}=\lambda_{n} e_{n}, \quad n\in\N,\] where the functions   $e_{n}=\sqrt{2}\sin(n\pi\cdot)$, $n\ge 1$,  form an orthonormal basis of $X$. With respect to this basis, the operator $B=a\delta_1'$ can be identified with $(b_{n})_{n\in\N}$ for $b_{n}=(-1)^n \sqrt{2}an\pi$, $n\in\N$. Therefore, 
\[\sum_{n\in\N} \frac{|b_{n}|^{2}}{|\lambda_{n}|^{2}}=\frac{1}{3}<\infty,\]
which shows that $B\in X_{-1}$. By Theorem \ref{theodiag}, we conclude that the system is $L^\infty$-iISS. Moreover, we obtain the following $L^\infty$-ISS and $L^\infty$-iISS estimates:
\begin{align*}
\|x(t)\|_{L^2(0,1)}&\le {\rm{e}}^{-a\pi^{2} t}\|x_0\|_{L^2(0,1)} +\frac{1}{\sqrt{3}} \|u\|_{L^\infty(0,t)},\\
\|x(t)\|_{L^2(0,1)}&\le {\rm{e}}^{-a\pi^{2} t}\|x_0\|_{L^2(0,1)} +c \left( \int_0^t |u(s)|^pds \right)^{1/p},
\end{align*}
for  $p> 2$ and some constant $c=c(p)>0$. For the second inequality, we used the fact that $\Sigma(A,B)$ is even $L^{p}$-admissible for $p> 2$, as it can be shown by applying Theorem 3.5 in  \cite{JPP14}. We note that a slightly weaker $L^\infty$-ISS estimate for this system can also be found in \cite{Krstic16MTNS}.
\end{example}

\begin{example}\label{ex27}
As remarked, Example \ref{ex1} provides a system $\Sigma(A,B)$ which is even $L^{p}$-admissible for $p>2$. In the following we present a system which is $L^\infty$-admissible, but not $L^p$-admissible for any $p<\infty$. In order to find such an example, we use the characterization of $L^{p}$-admissibility from \cite[Thm.~3.5]{JPP14}.\\
Let $X=\ell^{2}$ and let $(\lambda_{n})_{n\in\N}$, $(b_{n})_{n\in\N}$ define a parabolic diagonal system $\Sigma(A,B)$ as in Section \ref{sec:diag.syst}.
Furthermore, let $p\in(2,\infty)$. Then $\Sigma(A,B)$ is infinite-time $L^{p}$-admissible if and only if
\[ \left(2^{-\frac{2n(p-1)}{p}}\mu(Q_{n})\right)_{n\in\Z}\in\ell^{\frac{p}{p-2}}(\Z),\]
where $\mu=\sum_{n\in\Z}|b_{n}|^{q}\delta_{\lambda_{n}}$ and $Q_{n}=\{z\in\C \mid \Re z\in(2^{n-1},2^{n}]\}$, $n\in\Z$.\\
We choose $\lambda_{n}=-2^{n}$ and $b_{n}=\frac{2^{n}}{n}$ for $n\in\N$. Clearly, $B=(b_{n})\in X_{-1}$. Then we have that 
\[2^{-\frac{2n(p-1)}{p}}\mu(Q_{n})=2^{-\frac{2n(p-1)}{p}}\frac{2^{2n}}{n^{2}}=\frac{2^{\frac{2n}{p}}}{n^{2}},\]
and thus for $p>2$,
\[\left(\left(2^{-\frac{2n(p-1)}{p}}\mu(Q_{n})\right)^{\frac{p}{p-2}}\right)_{n\in\Z}=\left(\frac{2^{\frac{2n}{p-2}}}{n^{\frac{2p}{p-2}}}\right)_{n\in\Z}\notin\ell^{1}.\]
Hence, $\Sigma(A,B)$ is not $L^{p}$-admissible for any $p>2$, and therefore also not for any $p\geq1$. 
However, since $\sum_{n \in \mathbb N} |b_n|^2/ |\Re \lambda_n|^2 = \sum_{n \in \mathbb N} 1/n^2 < \infty$, Theorem \ref{theodiag} shows that $\Sigma(A,B)$ is $L^\infty$-iISS and, in particular infinite-time $L^{\infty}$-admissible.\newline
We observe that by Theorem \ref{thm:main}, there exists a Young function $\Phi$ such that $\Sigma(A,B)$ is $E_{\Phi}$-admissible. {However, as the system is not $L^{p}$-admissible, such $\Phi$ cannot satisfy the $\Delta_{2}$-condition, see Remark \ref{rem:delta2}.} 
\end{example}

\section{Conclusions and Outlook} \label{sec:Weissproblem}
In this paper, we have studied the relation between input-to-state stability and integral input-to-state stability  for linear infinite-dimensional systems with a  (possibly) unbounded control operator and inputs in general function spaces. {In this situation, ISS is equivalent to admissibility together with exponential stability of the semigroup.}
 We have related the notions of iISS with respect to $L^1$ and $L^\infty$ to ISS with respect to Orlicz spaces. The known result that ISS and iISS are equivalent for $L^{p}$-inputs with $p<\infty$, was generalized to Orlicz spaces that satisfy the $\Delta_{2}$-condition. 
Moreover, we have shown that for parabolic diagonal systems and scalar input, the notions of  $L^{\infty}$-iISS and $L^{\infty}$-ISS coincide. 

Among possible directions for future research are the investigation
of the non-analytic diagonal case, general analytic systems and the relation of zero-class admissibility and input-to-state stability.
{Recently, the results on parabolic diagonal systems have been adapted to more general situations of analytic semigroups -- the crucial tool being the holomorphic functional calculus for such semigroups \cite{JacSchwZwa}.
Furthermore, versions ISS and iISS for strongly stable semigroups rather than exponentially stable can be studied, see \cite{NabSchwe}.}

Finally, we mention that the existence of a counterexample for one of the unknown (converse) implications in Figure \ref{fig1} can be related to the following open question posed by G.~Weiss in \cite[Problem 2.4]{Weiss89ii}.
\medskip

\noindent
{\bf Question A:} {\it Does the mild solution $x$ belong to $C([0,\infty),X)$ for any $x_{0}\in X$ and $u\in Z=L^{\infty}(0,\infty;U)$ provided that $\Sigma(A,B)$ is $L^\infty$-admissible?}
\medskip

Although we do not provide an answer to this question, we relate it to 

\begin{proposition}
At least one of the following assertions is true.
\begin{enumerate}
\item The answer to Question A is positive for every system $\Sigma(A,B)$.
\item There exists a system $\Sigma(A_{0},B_{0})$, with $A_{0}$ generating an exponentially stable semigroup and  $\Sigma(A_{0},B_{0})$  is $L^\infty$-admissible, but not $L^\infty$-zero-class admissible.
\end{enumerate}
\end{proposition}
\begin{proof}
This follows directly from Proposition \ref{prop:zeroclass2}.
\end{proof}

\appendix
\section{Orlicz Spaces}
In this section we recall some basic definitions and facts about Orlicz spaces. More details can be found in \cite{KrasnRut,Kufner,Adams,Zaanen}. For the generalization to vector-valued functions see \cite[VII, Sec.~7.5]{RaoOrlicz}. 
In the following $I \subset \mathbb R$ is an open bounded interval, $U$ is a Banach space and $\Phi \colon \Rpn\to\Rpn$ is a function.
\begin{definition}
The \emph{Orlicz class} $\tilde{L}_\Phi(I;U)$ is the set of all equivalence classes (w.r.t.\ equality almost everywhere) of Bochner-measurable functions $u \colon I \to U$ such that
  \begin{equation*}
    \rho(u; \Phi) := \int_I \Phi(\|u(x)\|_{U}) \, dx < \infty.
  \end{equation*}
\end{definition}
In general, $\tilde{L}_\Phi(I;U)$ is not a vector space. Of particular interest are Orlicz classes generated by Young functions.
\begin{definition}\label{def:YF}
  A function $\Phi:[0,\infty)\rightarrow \mathbb R$ is called a \emph{Young function (or Young function generated by $\varphi$)} if
  \begin{equation*}
    \Phi(t) = \int_0^t \varphi(s) \, ds, \qquad t \geq 0,
  \end{equation*}
where the function $\varphi \colon [0, \infty) \to \mathbb R$ has the following properties:
$\varphi$ is right-continuous and 
 nondecreasing, $\varphi(0)= 0$, $\varphi(s) >0$ for $s>0$ and
$\lim_{s \to \infty}\varphi(s) = \infty$.
\end{definition}

\begin{theorem}[\protect{\cite[Thm.~3.2.3 and Thm.~3.2.5]{Kufner}}] \label{theo:l1orlicz}
  Let $\Phi$ be a Young function. Then $\tilde{L}_\Phi(I;U)$ is a
  convex set and $\tilde{L}_\Phi(I;U) \subset L^1(I;U)$. Conversely, for $u \in L^1(I;U)$  there is a Young function $\Phi$ such that $u \in \tilde{L}_\Phi(I;U)$.
\end{theorem}

\begin{definition}
 Let $\Phi$ be the Young function generated by $\varphi$. Then $\Psi$ defined by
\begin{equation*}
 \Psi(t) = \int_0^t \psi(s) \,ds \quad \mbox{ with}\quad  \psi(t) = \sup_{\varphi(s) \leq t} s, \quad t\geq0,
\end{equation*}
is called
the \emph{complementary function} to $\Phi$.
\end{definition}

The complementary function of a Young  function is again a Young function.
 If $\varphi$ is continuous and strictly increasing in $[0, \infty)$, { i.e.\ belongs to $\mathcal K_\infty$}, then $\psi$ is the inverse function $\varphi^{-1}$ and vice versa. We call $\Phi$ and $\Psi$ a \emph{pair of complementary Young functions}. 

\begin{theorem}[Young's inequality, \protect{\cite[Thm.~I, p.~77]{Zaanen}}]
  Let $\Phi$, $\Psi$ be a pair of complementary Young functions and $\varphi$, $\psi$ their generating functions. Then 
  \begin{equation*}
    u v \leq \Phi(u)+\Psi(v),\qquad \forall u, v \in [0, \infty).
  \end{equation*}
Equality holds if and only if $v = \varphi(u)$ or $u = \psi(v)$.
\end{theorem}
\begin{rem}
\label{rem:A:rhoOrlicz}
  Let  $\Phi$, $\Psi$ be a pair of complementary Young functions, $u \in \tilde L_\Phi (I)$ and $v \in \tilde L_\Psi (I)$. By integrating Young's inequality we get
  \begin{equation*}
    \int_I |u(x) v(x)| \,dx \leq \rho(u;\Phi)+\rho(v;\Psi)
  \end{equation*}
 \end{rem}
We are now in the position to define the Orlicz spaces for which several equivalent definitions exist. Here we use the so-called {\em Luxemburg norm.}

\begin{definition} For a Young function $\Phi$,
 the set $L_\Phi(I;U)$ of all equivalence classes (w.r.t.\ equality almost everywhere) of Bochner-measurable functions $u \colon I \to U$ for which there is a $k > 0$ such that
 \begin{equation*}
   \int_I \Phi(k^{-1} \|u(x)\|_{U}) \, dx < \infty
 \end{equation*}
is called the \emph{Orlicz space}. The \emph{Luxemburg norm} of $u \in L_\Phi(I;U)$ is defined as 
\begin{equation*}
     \|u\|_{\Phi}:= \|u\|_{L_{\Phi}(I;U)}: =\inf\left\{k>0 \Bigm| \int_I \Phi(k^{-1}\| u(x)\|) \,ds \leq 1 \right\}.
  \end{equation*}
\end{definition}
For the choice $\Phi(t):=t^p$, $1 < p<\infty$, the Orlicz space $L_\Phi(I;U)$ equals the vector-valued $L^p$-spaces with equivalent norms.
\begin{theorem}[\protect{\cite[Thm.~3.9.1]{Kufner}}]
  $(L_\Phi (I;U),  \|\cdot\|_\Phi)$ is a Banach space.
\end{theorem}
Clearly, $L^\infty(I, U)$ is a linear subspace of $L_\Phi(I, U)$.
\begin{definition}
The space $E_\Phi(I, U)$ is defined as  
\begin{equation*}
    E_\Phi(I, U) = \overline{L^\infty(I, U)}^{\|\cdot\|_{L_{\Phi}(I;U)}}.
  \end{equation*}
  The norm $\|\cdot\|_{E_{\Phi}(I;U)}$ refers to $\|\cdot\|_{L_{\Phi}(I;U)}$.
\end{definition}
If $U=\mathbb K$ with $\mathbb K \in\{\mathbb R,\mathbb C\}$, then we write $L_\Phi(I):=L_\Phi(I;\mathbb K)$ and $E_\Phi(I):=E_\Phi(I;\mathbb K)$ for short. The Banach spaces  $E_\Phi(I; U)$ and  $L_\Phi(I; U)$ have the following properties.
\begin{rem}\label{rem:orlicz}
  \begin{enumerate}
  \item $E_{\Phi}(I;U)$ is separable, see e.g.\ \cite[Thm.~6.3]{Schappacher}. \label{rem:orlicz-part1}
  \item For a measurable $u \colon I\to U$,  $u\in L_{\Phi}(I;U)$ if and only if  $f=\|u(\cdot)\|_U\in L_{\Phi}(I)$. This follows from the fact that $\|u\|_{\Phi}=\|f\|_{\Phi}$.
  Thus,  $(u_{n})_{n\in\N} \subset L_{\Phi}(I;U)$ converges to $0$ if and only if $(\|u_{n}(\cdot)\|_U)_{n\in\N}$ converges to $0$ in $L_{\Phi}(I)$. \label{rem:orlicz-part2}
\item { Let $\Phi$, $\Psi$ be a pair of complementary Young functions. The extension of H\"older's inequality to Orlicz spaces reads: For any $u\in L_{\Phi}(I)$ and $v\in L_{\Psi}(I)$, it holds that $uv\in L^{1}(I)$ and
$$ \int_{I}|u(s)v(s)|\,ds\leq 2\|u\|_{L_{\Phi}(I)}\|v\|_{L_{\Psi}(I)},$$
see 
\cite[Thm.~3.7.5 and Rem.~3.8.6]{Kufner}.} This implies that for $u\in L_{\Phi}(I;U)$,
  \begin{equation*}
     \|u\|_{L^1(0,t;U)} = \int_0^t \|u(s)\|_{U} \,ds \leq 2\|\chi_{(0,t)}\|_\Psi  \|u\|_\Phi,
  \end{equation*}
i.e., $L_\Phi(I;U)$ is continuously embedded in $L^1(I;U)$. Moreover, $\|\chi_{(0,t)}\|_\Psi \to0$ as $t \searrow 0$, where $\chi_{(0,t)}$ denotes the characteristic function of the interval $(0,t)$.\label{rem:orlicz-part3}
\item\label{rem:orlicz-part4} $E_{\Phi}(I;U) \subset \tilde{L}_{\Phi}(I;U)\subset L_{\Phi}(I;U),$
see e.g.\ \cite[Thm.~5.1]{Schappacher}. For $u\in \tilde{L}_\Phi(I;U)$, 
\begin{equation*}
  \|u\|_{\Phi}  \leq \rho(\|u(\cdot)\|_U;\Phi) +1 < \infty.
\end{equation*}
  \end{enumerate}
\end{rem}
\begin{definition}[$\Phi$-mean convergence]\label{def:meanconv}
  A sequence $(u_n)_{n \in \mathbb N}$ in $L_\Phi(I)$ is said to converge in $\Phi$-mean to $u \in L_\Phi(I)$ if
  \begin{equation*}
    \lim_{n \to \infty} \rho(u_n-u; \Phi) = \lim_{n \to \infty} \int_I \Phi(|u_n(x)-u(x)|) \, dx = 0.
  \end{equation*}
\end{definition}
\begin{definition}\label{def:deltatwo}
We say that a Young function $\Phi$ \textit{satisfies the $\Delta_{2}$-condition} if 
\[\exists k>0,u_{0}\geq0\ \forall u\geq u_{0}:\quad \Phi(2u)\le k \Phi(u).\]
\end{definition}

It holds that $ E_{\Phi}(I;U)= \tilde L_\Phi (I;U) = L_\Phi (I;U)$ if $\Phi$ satisfies the $\Delta_{2}$-condition.
\begin{definition}\label{def:increase}
  Let $\Phi$ and $\Phi_1$ be two Young functions. We say that the function $\Phi_1$ \emph{increases essentially more rapidly} than the function $\Phi$ if, for arbitrary $s>0$,
  \begin{equation*}
    \lim_{t \to \infty} \frac{\Phi(st)}{\Phi_1(t)} = 0.
  \end{equation*}
\end{definition}
{\begin{theorem}[\protect{\cite[Thm.~13.4]{KrasnRut}}]
 Let $\Phi, \Phi_{1}$ be Young functions such that $\Phi_{1}$ increases essentially more rapidly than $\Phi$. If $(u_n)_{n \in \mathbb N}\subset \tilde{L}_{\Phi_{1}}(I)$ converges to $0$ in $\Phi_{1}$-mean, then it also converges in the norm $\|\cdot\|_{\Phi}$.
\label{thm:increasemorerapidly}\end{theorem}}
\section*{Acknowledgments}
The authors would like to thank Andrii Mironchenko for valuable discussions 
on ISS. They also wish to express their gratitude to Jens Wintermayr for pointing out an error in a previous version. 

\def\cprime{$'$}

\end{document}